\documentclass[12pt,english]{article}
\usepackage[T1]{fontenc}
\usepackage[latin9]{inputenc}
\usepackage{amsmath, amsfonts, amsthm}
\usepackage{amssymb}
\usepackage{textcomp}

\newtheorem{assumption}{Assumption}
\newtheorem{lemma}{Lemma}
\newtheorem{theorem}{Theorem}
\newtheorem{corollary}{Corollary}
\newtheorem{remark}{Remark}

\newcommand{\s}{\mathcal{E}}
\newcommand{\ZZ}{\mathbb{Z}}

\newcommand{\wt}{\widetilde}
\newcommand{\wh}{\widehat}
\newcommand{\bm}{\boldsymbol}
\newcommand{\diag}{{\mathrm{diag}}}
\makeatletter

 \usepackage{xcolor}
\definecolor{mauve}{RGB}{0, 180, 0}

\makeatother

\usepackage{babel}
\addto\extrasfrench{      \providecommand{\fg}{\ifdim\lastskip>\z@\unskip\fi~\frqq}}

\begin{document}

\title{General solution of the Poisson equation for Quasi-Birth-and-Death processes}

\author{Dario Bini\footnote{Dipartimento di Matematica, Universit\`{a}
    di Pisa, Largo Bruno Pontecorvo, 5, 56127 Pisa,  Italy} \and
Sarah Dendievel\footnote{Ghent University, Department of Telecommunications and Information Processing, SMACS Research Group, Sint-Pietersnieuwstraat 41, B-9000 Gent, Belgium} \and
Guy Latouche\footnote{
Universit\'e libre de Bruxelles,
D\'epartement d'informatique, CP212, 1050 Bruxelles, Belgium} \and
Beatrice Meini$^{*}$}

\maketitle
\begin{abstract}
We consider the Poisson equation
$(I-P)\boldsymbol{u}=\boldsymbol{g}$, where
$P$ is the transition matrix of a Quasi-Birth-and-Death (QBD) process with infinitely many levels, $\bm g$ is a given infinite dimensional vector and $\bm u$ is the unknown. 
Our main result is to provide the general solution of this equation. To this purpose we use the block tridiagonal and block Toeplitz structure of the matrix $P$  to 
obtain a set of matrix difference equations, which are solved by constructing suitable resolvent triples.
\end{abstract}



\section{Introduction}
Given a row-stochastic matrix $P$ and a vector $\bm g$, the Poisson
equation is written as
\begin{equation}\label{PoissonEqGen}
(I-P)\bm u=\bm g,
\end{equation}
where $\bm u$ is the unknown vector. 
A matrix $P$ is row-stochastic if it has nonnegative entries and  $P\boldsymbol{1}=\boldsymbol{1}$, where $\boldsymbol{1}$ is the vector with components equal to 1.
The Poisson equation is
 important in Markov chain theory, where $P$ represents the transition probability matrix of an irreducible homogeneous Markov chain. Some examples of applications are: heavy-traffic limit theory e.g.  Asmussen
\cite{asmussen1992queueing}, central limit theorem e.g. Glynn
\cite{glynn1994poisson}, variance constant analysis e.g. Asmussen and
Bladt \cite{asmussen1994poisson}, mean and variance of mixing times of
QBD processes e.g. Li and Cao \cite{li2013mixing}, asymptotic variance
of single-birth process e.g. Jiang {\it{et al.}}
\cite{jiang2014poisson}.  

If the matrix $P=(p_{i,j})_{i,j=1,n}$ is finite and $\bm\pi$ is the stationary distribution of the Markov chain, i.e., $\bm\pi$ is the vector such that $\bm\pi^TP=\bm\pi^T$ and $\bm\pi^T\bm 1=1$,  then the Poisson equation has solutions if and only if $\bm\pi^T\bm g=0$. In fact, this condition is equivalent to the property that $\bm g$ belongs to the span of the columns of $I-P$. Moreover, 
 the Poisson equation
 has a unique solution, up to an additive constant, given by
\begin{equation}\label{eq:poif}
 \boldsymbol{u}=(I-P)^{\#}\boldsymbol{g}
 +\alpha \boldsymbol{1}
 \end{equation}
for an arbitrary  scalar $\alpha$, where  $H^{\#}$ denotes the group inverse of the matrix $H$
(see Meyer 
		\cite{meyer1975role}
and Campbell and Meyer
		 \cite{campbell2009generalized}).
In the case of infinite state Markov chains, 
$P=(p_{i,j})_{i,j\in\mathbb N}$ has infinite size, and
uniqueness  does not hold in general. This case has been studied by Makowski and Schwartz in
\cite{makowski2002poisson}, where the authors  give some characteristics of the
solutions.  

A class of infinite dimensional Markov chains, having great importance in applications, is given by Quasi-Birth-and Death processes.
 The transition matrix of these
processes has  the following block tridiagonal, almost block-Toeplitz structure
\begin{equation}\label{transitionMatrixP}
P=\begin{bmatrix}
B&A_1\\
A_{-1}&A_0&A_1\\
      &A_{-1}&A_0&\ddots\\
      &&\ddots&\ddots&
\end{bmatrix}
\end{equation}
where $B,A_{-1},A_{0}$ and $A_{1}$ are square matrices of order
$m<\infty$.  
The QBD processes are described and analyzed for instance
in Neuts \cite{neuts1981matrix} and Latouche and Ramaswami~\cite{latouche1999introduction}.  
%
%
In Dendievel {\it{et al.}}
\cite{dendievel2013poisson}, the authors derive particular solutions
of the Poisson equation for QBDs in terms of the deviation matrix, and
their solution is based on probabilistic arguments.  

In this paper, we
focus instead on finding all the solutions of the Poisson equation for a QBD, by exploiting  the structure of the matrix \eqref{transitionMatrixP}.  Indeed, rewriting \eqref{PoissonEqGen} in terms of
 the blocks of the transition matrix \eqref{transitionMatrixP} yields the set of  equations
\begin{align}
(B-I)\boldsymbol{u}_{0}+A_{1}\boldsymbol{u}_{1} 
&= - \boldsymbol{g}_{0}, 
\label{InitialCondEq}
\\
A_{-1}\boldsymbol{u}_{r}+(A_{0}-I)\boldsymbol{u}_{r+1}+A_{1}\boldsymbol{u}_{r+2}
&= -	\boldsymbol{g}_{r+1}, 
\label{matrixequation}
\end{align}
for $r\geq0$, where the infinite vectors $\boldsymbol{u}$ and $\boldsymbol{g}$ have been partitioned into blocks $\boldsymbol{u}_i$, $\boldsymbol{g}_i$, $i\ge0$, of length $m$.  
Equation \eqref{matrixequation}, for $r\ge 0$, represents a matrix difference equation, while \eqref{InitialCondEq} provides the initial condition.
Seen in this way, the Poisson equation for QBDs may be analyzed by relying on the theory of matrix difference equations developed by Gohberg {\it{et al.}} in \cite{glr11b}.

We  provide the general expression of the solution of the matrix difference equation~\eqref{matrixequation} by means of resolvent triples of the matrix polynomial
 \begin{equation}\label{matrixpoly}
\eta(\lambda)=
A_{-1}
+(A_0-I) \lambda
 +A_1 \lambda^{2}.
\end{equation}
 Due to the stochasticity of $P$, the polynomial $\det(\eta(\lambda))$ has a root $\lambda$ equal to $1$.
If $\lambda=1$ is a simple zero, like for positive recurrent or transient QBDs, then we may construct a resolvent triple from the minimal nonnegative solutions $G$ and $\wh G$ of the quadratic matrix equations $A_{-1}+ (A_0-I)X+ A_1 X^2=0$ and $A_{-1} X^2+ (A_0-I) X+ A_1=0$, respectively. If the QBD is
null recurrent, then  $\lambda=1$ is not a simple zero. In this case
we cannot construct a resolvent triple directly form $G$ and $\wh G$,
and we follow a different approach which consists in transforming the
matrix difference equation ~\eqref{matrixequation} into a modified difference equation
 $\wt A_{-1}\wt{\boldsymbol{u}}_{r}+(\wt A_{0}-I)
 \wt{\boldsymbol{u}}_{r+1}+
 \wt A_{1}\wt{\boldsymbol{u}}_{r+2}
= -\wt{\boldsymbol{g}}_{r+1}$, 
such that the polynomial $\det(\wt A_{-1}+(\wt A_0-I)\lambda +\wt A_1 \lambda^2 )$ has a simple root at $\lambda=1$. This new difference equation can be solved by means of resolvent triples and we give an explicit expression relating the solutions of the modified difference equation to the solutions of the original equation.
This transformation relies on the shift technique introduced by He {\it{et al.}} \cite{hmr} and developed by Bini {\it{et al.}} in \cite{blm:shift}.

Once we have a general expression of the solution of the difference equation \eqref{matrixequation}, we impose the initial condition \eqref{InitialCondEq}. In the positive and null recurrent cases, the initial condition leads to a Poisson equation of finite size, which can be solved by means of the group inverse according to equation \eqref{eq:poif}. In the transient case, the initial condition leads to a nonsingular linear system. In all cases, the expression of the general solution of the Poisson equation depends on an arbitrary vector $\bm y$. We show that
the particular 
solution obtained in Dendievel {\it{et al.}} \cite{dendievel2013poisson} by means of probabilistic arguments, corresponds to a specific choice of the vector $\bm y$.

The paper is organized as follows.  
In Section  \ref{SectionPrelim}
we recall the fundamental elements on QBD processes and introduce some key matrices.
In Section \ref{SectionResolv} we analyze the matrix difference
equation adapted to our problem and introduce the notion of resolvent
triple.  
The main results of this paper are given in Sections
\ref{SectionSolution} and \ref{s:null}. In Section
\ref{SectionSolution}, Theorem~\ref{mainresult} provides the
general solution of the Poisson equation in the case
of a positive recurrent or a transient QBD process. In Section~\ref{s:null} we deal with the null recurrent case 
and show in details two different approaches based on the shift technique.
We compare in Section \ref{SectionComp} the particular probabilistic
solution given in Dendievel {\it{et al.}}  \cite{dendievel2013poisson}
with the solution given in Section~\ref{SectionSolution}.


\section{Quasi-Birth-and-Death process }\label{SectionPrelim}

Some properties on QBD processes will be used in the next sections.
We consider a discrete-time QBD process with transition matrix $P$ given in
(\ref{transitionMatrixP}),
on the  state space 
$\mathcal S=\{(n,i):n\in\mathbb{N},i\in \mathcal E\},
\mathcal E=\{1,\ldots,m\}$,
where $n$ is called the level
and $i$ is the phase.
We define the matrix $G$
as the minimal nonnegative solution of the equation
\begin{equation}\label{Geq}
A_{-1}+(A_{0}-I)X+A_{1}X^{2}=0,
\end{equation}
and
the matrix $\wh{G}$
as the minimal nonnegative solution of the equation
\begin{equation}\label{Ghateq}
A_{1}+(A_{0}-I)X+A_{-1}X^{2}=0.
\end{equation}
The component $G_{ij}$
of the matrix $G$ is
 the conditional probability that the process goes to the level $n$ in a  finite time and that $j$ is the  first phase visited in this level, given that the process starts from state $(n + 1, i)$,
for $i,j\in \mathcal E$,
$n\in\mathbb{N}$.
The matrix $\wh{G}$ corresponds to the matrix $G$ of the level-reversed process.

Throughout the paper we assume that $P$ is irreducible, that $A_{-1}+A_0+A_1$ is irreducible and that the following property holds.
\begin{assumption}
   \label{a:one}
The doubly infinite QBD on $\ZZ \times \s$ has only one final class
$\ZZ \times \s_*$, where $\s_* \subseteq \s$.  Every other state is on a
path to the final class.
Moreover, the set $\s_*$ is not empty.
\end{assumption}

Assumption \ref{a:one} is Condition 5.2 in \cite[Page 111]{blm:book} where it is implicitly assumed that 
$\s_*$ is not empty.

The roots of the polynomial $\phi(\lambda)=\det \eta(\lambda)$, where
$\eta(\lambda)$ is defined in \eqref{matrixpoly},
 have a useful
property that we give now for future reference (Bini {\it et al.} \cite[Theorem 4.9]{blm:book}, Govorun {\it et al.}
\cite[Theorem 3.2]{glr11b}).

\begin{lemma}
   \label{t:roots}
Denote by $\xi_1, \ldots, \xi_{2m}$ the roots of $\phi(\lambda)$, organized
so that 
$|\xi_1| \leq |\xi_2| \leq \cdots \leq |\xi_{2m}|$, with $\xi_{d+1} =
\cdots = \xi_{2m} = \infty$ if the degree of $\phi(\lambda)$ is $d < 2m$.
The roots $\xi_1$, \ldots, $\xi_m$ are the eigenvalues of the matrix
$G$ and $\xi_{m+1}$, \ldots, $\xi_{2m}$ are the reciprocals of the
eigenvalues of $\wh G$ with the convention that $1/\infty=0$ and $1/0=\infty$.  The roots $\xi_m$, $\xi_{m+1}$ are real and one has 
\[|\xi_{m-1}| < \xi_m \leq 1 \leq \xi_{m+1} < |\xi_{m+2}|.
\]
Moreover 
\begin{itemize}
\item 
if the QBD is positive recurrent then
$\xi_m = 1 < \xi_{m+1}$, $G$ is stochastic and $\widehat G$ is sub-stochastic,
\item 
if the QBD is transient then $\xi_m < 1 = \xi_{m+1}$, $G$ is sub-stochastic  and
$\widehat G$ is stochastic,
\item 
if the QBD is null recurrent then $\xi_m = 1 = \xi_{m+1}$,  $G$ and $\widehat G$ are stochastic.
\end{itemize}
\end{lemma}

Furthermore, thanks to the repetitive structure of $P$,
the stationary distribution 
$\boldsymbol{\pi}$ of the process,
partitioned  as $\boldsymbol{\pi}=(\boldsymbol{\pi}_i)_{i\ge 0}$,
where 
$\boldsymbol{\pi}_i$
is an $m$-dimensional vector representing the stationary probability of the level $i$, has a matrix-geometric
structure, that is,
 $\boldsymbol{\pi}_i^T=\boldsymbol{\pi}_0^T R^{i}$,
 where $R$ is the minimal nonnegative solution of the equation
\begin{equation}\label{Req}
A_1+X(A_{0}-I)+X^{2}A_{-1}=0
\end{equation}  (Latouche and Ramaswami~\cite[Theorems 6.2.1 and  6.2.10]{latouche1999introduction}).

We assume in this section and the next two that the QBD process is
positive recurrent or transient ({\em not null recurrent} for short).
 Under this assumption the series 
\begin{equation}\label{MatrixH0}
  W=\sum_{j=0}^{\infty}G^{j}(U-I)^{-1}R^{j}
  \end{equation} 
is convergent, where $U=A_0+RA_{-1}$, and it is shown in \cite{blm:shift} that the matrices $R$ and $\wh{G}$ are related by  
\begin{equation}\label{RelGandR}
WR=\wh{G}W.
\end{equation} 
Note for later reference that from \cite[Theorem 6.2.9]{latouche1999introduction} we have the relations
\begin{align}
R&=A_{1}(I-U)^{-1},
\label{RfctU}
\\
U&=A_{0}+A_1 G.
\label{UfctG}
\end{align}
In addition,
the vector $\bm \pi_0$ is a solution of 
\begin{equation}\label{e:pi0}
\bm \pi_0^{T}(I-B-A_1 G)=\bm 0,
\end{equation}
normalized by
$\bm \pi_0^{T}(I-R)^{-1}\bm 1=1$.

We show in Lemma~\ref{LemmaInVDbl} that $W$ is
invertible, so that $R$ and $\wh G$ are actually similar matrices.
The subsequent results give additional characterizations of the matrix $W$.

\begin{lemma}\label{LemmaInVDbl}
If the QBD is not null recurrent, then the matrix $W$ defined in
\eqref{MatrixH0}  and the matrix $G\wh G-I$ are nonsingular, moreover
\begin{equation}\label{eq:WGGhat}
W^{-1}= (I-U)\left(G\wh{G}-I\right).
\end{equation}
\end{lemma}
\begin{proof}
By  (\ref{RelGandR}), we obtain
\begin{align*}
\left(G\wh{G}-I\right)W  =  \sum_{j=0}^{\infty}G^{j+1}(U-I)^{-1}R^{j+1}-W
  =  - (U-I)^{-1}.
\end{align*}
It follows that $W$ and $G\wh G-I$ are nonsingular so that \eqref{eq:WGGhat} holds.
\end{proof}

\begin{lemma}
If the QBD is not null recurrent, then the matrix $W$ of
\eqref{MatrixH0} is such that
\begin{equation}\label{Charact1H0}
WA_{1}(G\wh{G}-I) =  \wh{G}.
\end{equation}
\end{lemma}

\begin{proof}
Since the QBD is not null recurrent, $W$ is well defined and
invertible, and it follows from (\ref{RelGandR}) that equation (\ref{Charact1H0}) is equivalent to 
\begin{equation*}
WA_{1}(G\wh{G}-I)W =  WR,
\text{\quad \quad or to \quad\quad} 
A_{1}GWR-A_{1}W=R.
\end{equation*}
Replacing $W$ by its definition leads to 
\begin{equation*}
A_{1}\sum_{j=0}^{\infty}G^{j+1}(U-I)^{-1}R^{j+1}
-A_{1}(U-I)^{-1}
-A_{1}\sum_{j=1}^{\infty}G^{j}(U-I)^{-1}R^{j}
=R,
\end{equation*}
which simplifies to
$
A_{1}(I-U)^{-1}=R,
$
a true relation given by \eqref{RfctU}.
\end{proof}

We introduce some notation.
Let $M$ be a nonsingular 
 matrix such that 
\begin{equation}\label{JordanDecompGHat}
\wh{G}M=MJ,
\end{equation}
with
\begin{equation}\label{partitionJordan}
J=\left[\begin{array}{cc}
V_1 & 0\\
0 & V_0
\end{array}\right],
\end{equation}
 where $V_1$ is a nonsingular square matrix of order $p$, $0\leq p \leq m$,
and $V_0$ is square matrix of order $m-p$ 
with $\rho(V_0)=0$.
For instance, we may choose $M$ such that $J$ is the Jordan normal form of $\wh{G}$, with $V_1$ containing all the blocks for the non-zero eigenvalues and $V_0$ containing all the blocks for the zero eigenvalues.
The matrix  $M$ may be written with corresponding dimensions
as 
\begin{equation}\label{partitionM}
M=\left[ L \;  \text{\rm{\textbrokenbar}} \; K\right],
\end{equation}
here, $L$ is a matrix with dimensions $m\times p$
and $K$ is a matrix with dimensions $m \times (m-p)$.
As a consequence, we have that (\ref{JordanDecompGHat})
may be equivalently written as the system of equations
\begin{equation}\label{GhKW}
\wh{G}L  = LV_1, 
\qquad
\wh{G}K  =  KV_0. 
\end{equation}
As a corollary of \eqref{Ghateq}, we have
\begin{equation}\label{CorollDV}
\begin{split}
A_{1}L+(A_{0}-I)LV_1+A_{-1}LV_1^{2}  &= 0, \\
A_{1}K+(A_{0}-I)KV_0 + A_{-1}KV_0^{2} &= 0. 
\end{split}
\end{equation}
The next proposition will be useful in Section \ref{SectionResolv}.
\begin{lemma}\label{Lemma33}
Assume that the QBD is not null recurrent.
Take
\begin{equation}\label{eq:Y}
Y  =  
\left[ A_{1}LV_1^{-1} \;  \text{\rm{\textbrokenbar}} \; -A_{-1}KV_0-(A_{0}-I)K\right],
\end{equation}
with $M,L,K,V_1$ and $V_0$ as defined above.
The matrix $Y-A_{1}GM$ is nonsingular and the matrix $W$ defined in \eqref{MatrixH0}  is equal to
\begin{equation}\label{3rdCaract}
W=M\left(A_{1}GM-Y\right)^{-1}.
\end{equation}
\end{lemma}
\begin{proof}
We prove that 
\begin{equation}
   \label{e:aa}
W(A_{1}GM-Y)=M,
\end{equation}
from which the nonsingularity of $A_{1}GM-Y$ and \eqref{3rdCaract} follow, since $W$ and $M$ are nonsingular.
By Lemma \ref{LemmaInVDbl},  (\ref{e:aa}) will follow from
\begin{equation}\label{EquivEqforY}
Y=\left(I-U\right)\left(I-G\wh{G}\right)M+A_{1}GM.
\end{equation}
By  \cite[Theorem 6.2.9]{latouche1999introduction} and \eqref{JordanDecompGHat}, 
\begin{align*}
\left(I-U\right) & (I-G\wh{G})M+A_{1}GM  \\
 & =  \left(I-\left(A_{0}+A_{1}G\right)\right)\left(M-GMJ\right)+A_{1}GM,\\
  &= M-A_{0}M+\left(\left(A_{0}-I\right)G+A_{1}G^{2}\right)MJ, \\
 & =  M-A_{0}M-A_{-1}MJ,  \qquad \qquad \mbox{by \eqref{Geq}}\\
&=\left[ (I-A_{0})L-A_{-1}LV_1
 \;  \text{\rm{\textbrokenbar}} \; 
(I-A_{0})K-A_{-1}KV_0 \right],
\end{align*}
since  $M$ and $J$ may be replaced, respectively with \eqref{partitionM} and \eqref{partitionJordan}. 
From \eqref{CorollDV}, it follows that
\[
(I-A_{0})L-A_{-1}LV_1=A_{1}LV_1^{-1}
\]
so that  \eqref{EquivEqforY} is satisfied.
\end{proof}


\section{Resolvent triple}\label{SectionResolv}
We report from Gohberg {\it et al.} \cite{gohberg2009matrix} some
definitions and results concerning the resolution of matrix difference
equations. We apply these results to the solution of the Poisson
equation.

Given the $m\times m$ matrix polynomial $B(\lambda)=\sum_{i=0}^l B_i\lambda^i$ of degree $l$, a pair of matrices $(X,T)$, with $X$ of size $m\times ml$ and $T$ of size $ml\times ml$, is called a {\em decomposable pair} for $B(\lambda)$ if:
\begin{enumerate} 
\item $X=\left[ X_{1}
		 \;  \text{\rm{\textbrokenbar}} \; 
		 X_{2}\right]$,
and
$T=\left[\begin{array}{cc}
T_1 & 0\\
0 & T_{2}
\end{array}\right],	 
	$	 
where $X_1$ is an $m\times q$ matrix, $T_1$ is a $q\times q$ matrix, for some $0\le q\le ml$;
\item
the matrix
\[\left[\begin{array}{cc}
X_1 & X_2T_2^{l-1}\\
X_1T_1 & X_2T_{2}^{l-2}\\
\vdots & \vdots\\
X_1T_1^{l-1} & X_2
\end{array}\right]	 
\]
 is nonsingular;
 \item
 $\sum_{i=0}^l B_i X_1 T_1^i=0$ and  $\sum_{i=0}^l B_i X_2 T_1^{l-i}=0$. 
\end{enumerate}
Furthermore, the triple $(X,T,Z)$ is a {\it{resolvent triple}} of $B(\lambda)$ if $(X,T)$ is a decomposable pair of $B(\lambda)$ and $Z$ is a matrix such that 
 $B^{-1}(\lambda)=XT^{-1}(\lambda)Z$, where
$T(\lambda)=\diag(\lambda I - T_{1}, \lambda T_{2}-I)$.

 We state next a finite difference equation theorem
for general matrix equations
(\cite[Theorem 8.3]{gohberg2009matrix}).
\begin{theorem}\label{MDET} 
 Let $(X,T,Z)$ be a resolvent
triple
of the $m\times m$ matrix polynomial
$B(\lambda)=\sum_{i=0}^{l}B_{i}\lambda^{i}$,
 where 
$X=\left[ X_{1}
		 \;  \text{\rm{\textbrokenbar}} \; 
		 X_{2}\right]$,
and
$T=\diag{(T_1,T_2)}$,
and let
$Z=
\mbox{\tiny $
\left[\begin{array}{cc}
Z_1\\
\\
Z_2
\end{array}\right]
$}
$
be the corresponding partition of $Z$.
The general solution of the homogeneous difference equation
\[
	B_{0}\boldsymbol{u}_{r}
	+B_{1}\boldsymbol{u}_{r+1}
	+\cdots
	+B_{l}\boldsymbol{u}_{r+l}
	=0,
\]
is
$
\boldsymbol{h}_{r}=X_1T_1^r\boldsymbol{z}
$, for $r\ge 0$, where $\boldsymbol{z}\in\mathbb{C}^{q}$
is arbitrary.
 
Let $\{\boldsymbol{f}_{r}\}_{r\in\mathbb{N}}$ be a sequence of vectors in $\mathbb{\mathbb{C}}^{m}$.
A particular solution of the non-homogeneous difference equation
\begin{equation}\label{GeneralEq}
	B_{0}\boldsymbol{u}_{r}
	+B_{1}\boldsymbol{u}_{r+1}
	+\cdot\cdot\cdot
	+B_{l}\boldsymbol{u}_{r+l}
	=\boldsymbol{f}_{r},
\end{equation}
is given by
\[
\boldsymbol{\sigma}_{r}
	=
	-\sum_{i=0}^{\nu-1}X_{2}T_{2}^{i}Z_2\boldsymbol{f}_{i+r}
	+\sum_{j=0}^{r-1} X_{1}T_{1}^{r-j-1}Z_{1}\boldsymbol{f}_{j}, ~~r\ge 0,
\]
for some positive integer $\nu$ such that $T^{\nu}_{2}=0$.
The general solution of the non-homogeneous equation \eqref{GeneralEq} is
\[
\boldsymbol{u}_r=\boldsymbol{h}_r+\boldsymbol{\sigma}_r,\qquad r\ge 0.
\]
\end{theorem}

In our QBD problem, 
(\ref{GeneralEq})
reduces to 
(\ref{matrixequation})
and the matrix polynomial $\eta(\lambda)$  defined in
(\ref{matrixpoly})  plays the role of the matrix polynomial
$B(\lambda)$ defined in Theorem \ref{MDET}.

We use known properties of the blocks of the transition matrix to
construct a resolvent triple of $\eta(\lambda)$ and we obtain in the
next section a general solution of the Poisson equation.  Equation
(\ref{InitialCondEq}) furnishes a supplementary condition on the
vector called $\boldsymbol{z}$ in Theorem \ref{MDET}.  First, we give 
a decomposable pair of  $\eta(\lambda)$ in the following lemma.
In its proof, we find it helpful to indicate explicitly the dimensions of the identity matrix and in such cases we indicate it as an index.

\begin{lemma}\label{decompair}
Assume that the QBD \eqref{transitionMatrixP} is not null recurrent.
Let $G$ be the minimal nonnegative solution of \eqref{Geq}, let $\wh
G$ be the minimal nonnegative solution of \eqref{Ghateq} and let
$L$, $K$, $V_1$ and $V_0$ be defined as in
\eqref{JordanDecompGHat}--\eqref{partitionM}. 

Define $X=\left[ X_{1}
		 \;  \text{\rm{\textbrokenbar}} \; 
		 X_{2}\right]$
 with 
	$X_{1}=[I_m  \;  \text{\rm{\textbrokenbar}} \; L]$,
	 $X_{2}=K$,
 and define 
 $T=\diag(T_1,T_2)$, with
	$T_{1}=\diag(G,V_1^{-1})$, $T_{2}=V_0$.

The pair $(X,T)$  is a decomposable pair of $\eta(\lambda)$.
\end{lemma}
\begin{proof}
We check the conditions (i), (ii) and (iii) of the definition of decomposable pairs.
Conditions (i) and (iii) are obvious by construction. For condition (ii), we have to verify that the matrix
	\begin{align*}
	\left[\begin{array}{cc}
X_{1} & X_{2}T_{2}\\
X_{1}T_{1} & X_{2}
\end{array}\right]
&=
\left[
\begin{array}{c|cc}
I_m & L & KV_0 \\
\hline
G & L V_1^{-1} & K
\end{array}
\right]
 = 
\left[
\begin{array}{c|c}
I_m & M J \\
\hline
G & M
\end{array}
\right]
\left[
\begin{array}{c|cc}
I_m &0&0\\
\hline
0 & V_1^{-1}&0  \\
0 & 0& I_{m-p}
\end{array}
\right],
\end{align*}
is nonsingular.
By \eqref{JordanDecompGHat}, we write
	\begin{align*}
	\left[\begin{array}{cc}
X_{1} & X_{2}T_{2}\\
X_{1}T_{1} & X_{2}
\end{array}\right]
 = 
\left[
\begin{array}{c|c}
I_m & \wh{G} \\
\hline
G & I_m
\end{array}
\right]
\left[
\begin{array}{c|c}
I_m & 0 \\
\hline
0 & M
\end{array}
\right]
\left[
\begin{array}{c|cc}
I_m & 0&0\\
\hline
0 & V_1^{-1} & 0 \\
0 & 0 & I_{m-p}
\end{array}
\right],
\end{align*}
and this is a product of nonsingular matrices.  In fact, the first
factor is nonsingular since its determinant coincides with the
determinant of $I-G\wh G$, which is nonsingular 
in view of Lemma \ref{LemmaInVDbl}.  The other two factors are
nonsingular by construction.
\end{proof}

Given a decomposable pair $(X,T)$ of a matrix polynomial, Theorem 7.7
in Gohberg {\it{et al.}} \cite{gohberg2009matrix} gives an explicit
expression for a matrix $Z$ such that $(X,T,Z)$ is a resolvent triple
for the same matrix polynomial.    In the next theorem, we 
adapt this directly to our special case where the matrix polynomial
is quadratic.  The construction of such a triple $(X,T,Z)$ will help
us to build the solution of the Poisson equation, relying on
Theorem \ref{MDET}.

We partition the inverse of the matrix $M$ defined in \eqref{JordanDecompGHat} as
\begin{equation}\label{PartitionMm1}
M^{-1}=
\left[\begin{array}{c}
E\\
F
\end{array}\right],
\end{equation}
where $E$ is $p\times m$
and $F$ is $(m-p)\times m$.

\begin{theorem}\label{ThmforZ}
Assume that the QBD \eqref{transitionMatrixP} is not null recurrent.
Let $Z$ be the matrix defined by
\[
Z=
\begin{bmatrix}Z_1\\Z_2\end{bmatrix},\quad Z_1=\begin{bmatrix}W\\-EW\end{bmatrix},\quad Z_2=-V_0FW,
\]
where $V_0$ is given in \eqref{partitionJordan}, and $E$ and $F$ are defined in \eqref{PartitionMm1}.
The triple $(X,T,Z)$, where $X$ and $T$ are given in Lemma \ref{decompair},  
 is a resolvent triple of  $\eta(\lambda)$.
\end{theorem}

\begin{proof}
The pair $(X,T)$ in Lemma \ref{decompair} is a decomposable pair of $\eta(\lambda)$.
By  \cite[Theorem 7.7]{gohberg2009matrix},
$(X,T,Z)$ is a resolvent triple of  $\eta(\lambda)$ if $Z$ takes the form
\begin{equation}\label{ZofhteThm77}
Z=\left[\begin{array}{cc}
I_{m+p} & 0\\
0 & V_0
\end{array}\right]
\Gamma^{-1}
\left[\begin{array}{c}
0\\
I_m
\end{array}\right],
\end{equation}
 where
%
\[
\Gamma 
=\left[\begin{array}{c|cc}
I_m&L&K\\\hline
A_1G&A_1LV_1^{-1}&-A_{-1}KV_0-(A_{0}-I)K
\end{array}\right]
=
\left[\begin{array}{cc}
I_m & M\\
A_{1}G & Y
\end{array}\right],
\]
%
with $Y$ defined in \eqref{eq:Y}.
The matrix $S= Y - A_{1}G M$ is the Schur complement of $I_m$ in
$\Gamma$. By Lemma \ref{Lemma33}, the matrix $S$ is invertible and $S=
-W^{-1}M$.   We have
\begin{align*}
\Gamma^{-1}\left[\begin{array}{c}
0\\
I_m
\end{array}\right]
&=
\left[\begin{array}{c}
-MS^{-1}\\
S^{-1}
\end{array}\right] 
=\left[\begin{array}{c}
W\\
-M^{-1}W
\end{array}\right]
=
\left[\begin{array}{c}
W\\
-EW\\
-FW
\end{array}\right]
\end{align*}
by \eqref{PartitionMm1}.  Replacing this in \eqref{ZofhteThm77}
completes the proof.
\end{proof}


\section{The general solution: non null recurrent case}
  \label{SectionSolution}

We need to recall the concept of  
group inverse of a matrix.
When it exists, the group inverse $H^{\#}$  of a square matrix $H$ is the matrix solving the three equations
$HH^{\#}=H^{\#}H$,
$HH^{\#}H=H$,
$H^{\#}HH^{\#}=H^{\#}$;
if $H$ is nonsingular, then $H^\# = H^{-1}$.
For an irreducible finite Markov process with transition matrix $P$,
the group inverse of the matrix $H=I-P$
 always exists, it is uniquely characterized by
 the set of equations
 \begin{equation}\label{eq:groupinv}
 I-(I-P)(I-P)^{\#}=\boldsymbol{1 \pi}^T,\quad \boldsymbol{\pi}^T (I-P)^{\#}=\bm 0
 \end{equation}
 where
 $\boldsymbol{\pi}$ is the stationary distribution vector of the
 Markov process
(see Theorem 8.5.5 in \cite{campbell2009generalized}).
As indicated in the introduction, if $\boldsymbol{g}$
 belongs to the columns  span of $I-P$ then the equation $(I-P) \boldsymbol{u}=\boldsymbol{g}$ has the solution 
\eqref{eq:poif}.

Relying on the results of the previous section, we provide an explicit
representation for the general solution of the Poisson equation in the
case of a non null recurrent QBD.   Under
this assumption, the matrix $W$ in \eqref{MatrixH0} exists and is
nonsingular, moreover, by Theorem \ref{ThmforZ} there exists a
resolvent triple $(X,T,Z)$ of $\eta(\lambda)$.
In the case of a homogeneous equation, the next lemma provides the
general solution, it is an immediate consequence of Theorems
\ref{MDET} and \ref{ThmforZ}.

\begin{lemma}
Let $G$ be the minimal nonnegative solution of \eqref{Geq}, let $L$ and $V_1$ be the matrices of size $m\times p$ and $p\times p$, respectively, defined through \eqref{JordanDecompGHat}, \eqref{partitionJordan}, \eqref{partitionM}.
The general solution of the homogeneous equation
\begin{align*}
A_{-1}\boldsymbol{u}_{r}+(A_{0}-I)\boldsymbol{u}_{r+1}+A_{1}\boldsymbol{u}_{r+2}
&= \bm 0, 
\end{align*}
 is given by
\[
\boldsymbol{h}_{r}=
G^r \bm x + LV_1^{-r}\bm y,\quad r\geq 0,
\]
where
$\bm x\in\mathbb{C}^{m}$
and
$\bm y \in \mathbb{C}^{p}$
are arbitrary.
\end{lemma}

The following result provides a particular solution of the non-homogeneous
equation together with the general solution.  It immediately follows
from Theorems \ref{MDET} and \ref{ThmforZ}.

\begin{lemma}\label{Lemma52}
Let $G$ be the minimal nonnegative solution of \eqref{Geq}, let $L$, $K$, $V_1$ and $V_0$ be the matrices of size $m\times p$, $m\times(m-p)$, $p\times p$ and $(m-p)\times(m- p)$, respectively, defined through \eqref{JordanDecompGHat}, \eqref{partitionJordan}, \eqref{partitionM}. Let $E$ and $F$ be the matrices defined in \eqref{PartitionMm1}, let $W$ be defined in \eqref{MatrixH0}.
A particular solution of 
\begin{align}\label{gensol1}
A_{-1}\boldsymbol{u}_{r}+(A_{0}-I)\boldsymbol{u}_{r+1}+A_{1}\boldsymbol{u}_{r+2}
&= -	\boldsymbol{g}_{r+1}, 
\end{align}
 is  given by
\begin{equation}\label{sigmar}
\boldsymbol{\sigma}_{r} = -\sum_{k=1}^{r}
\left(
G^{r-k}-LV_1^{k-r}E 
\right)W\boldsymbol{g}_{k}
-
\sum_{j=1}^{\nu-1}
KV_0^{j}FW\boldsymbol{g}_{j+r},\quad r\geq 0,
\end{equation}
where $\nu$ is the smallest integer such that $V_0^{\nu}=0$.  The
general solution of \eqref{gensol1} 
is 
\begin{align}\label{gensol2}
\boldsymbol{u}_{r}&=
G^r \bm x + LV_1^{-r}\bm y
+\boldsymbol{\sigma}_{r},
\quad r\geq 0,
\end{align}
where
$\bm x\in\mathbb{C}^{m}$
and
$\bm y \in \mathbb{C}^{p}$
are arbitrary.
\end{lemma}

Lemma \ref{Lemma52} characterizes all the solutions of  the difference
equation \eqref{matrixequation}. 
If we consider also the boundary condition \eqref{InitialCondEq}, we arrive at the following result, which expresses the general solution of the Poisson equation.

In the positive recurrent case, we need to assume that the series $\sum_{k=0}^{\infty}R^k \bm g_k$ converges. As $\rho(R)<1$ for positive recurrent QBDs, this allows some flexibility for asymptotic properties of the $\bm g_k$s. 

\begin{theorem}\label{mainresult}
The general solution of the Poisson equation
(\ref{PoissonEqGen}) is given by
\begin{equation}
\boldsymbol{u}_{r}
=
G^{r} \boldsymbol{x}
+
LV_1^{-r} \boldsymbol{y}
+
\boldsymbol{\sigma}_{r},\quad r\geq 0,
 \label{OurSolution}
\end{equation}
where $\boldsymbol{\sigma}_r$ is defined in \eqref{sigmar}
and
 $\bm x$ and $\bm y$
 satisfy the following constraints.
 
If the QBD is transient, then 
$\bm y \in \mathbb{C}^{p}$
is arbitrary and 
\begin{equation}
\boldsymbol{x}
= (I-P_{*})^{-1}
\left(
\left((B-I)\wh G + A_1\right)
\left( \bm \sigma_{1} + LV_1^{-1}\bm y
\right)
+\bm g_0
\right)
\label{e:xtran}
\end{equation}
where   $P_*  =B+A_1G$ 
and
\begin{align}
\bm \sigma_{1}
&=    
-\sum_{j=0}^{\nu-1}
KV_0^{j}FW\boldsymbol{g}_{j+1},
\label{defdblu}
\end{align}
with $\nu$ being  the smallest positive integer such that $V_0^{\nu}=0$.

If the QBD is positive recurrent and if the series $\sum_{k=0}^{\infty}R^k \bm g_k$ converges,
then 
\begin{align}
\bm y
&=
\bm y^{*} + \bm y_{\perp} 
\label{ExpressY}
\\
\boldsymbol{x}
&= (I-P_{*})^{\#}
\left(
\left((B-I)\wh G + A_1\right)
\left( \bm \sigma_{1} + LV_1^{-1}\bm y
\right)
+\bm g_0
\right)
+\alpha \bm 1
\label{ExpressX}
\end{align}
where 
\begin{equation}\label{eq:y}
\boldsymbol{y}^*=-\sum_{k=1}^\infty V_1^kEW\boldsymbol{g}_k,
\end{equation}
the vector 
$\bm y_{\perp}\in \mathbb{C}^{p}$
is any vector in the hyperplane 
$\bm \pi_{0}^{T} W^{-1} L \bm y = \bm \pi^{T}\bm g$,
and 
$\alpha$ is an arbitrary constant.

\end{theorem}

\begin{proof}
Firstly, we show that $\bm \sigma_{1}$ is given by \eqref{defdblu}: from \eqref{sigmar}, we have 
\begin{align}
\bm \sigma_1
&= -(I-LE)W \bm g_1
- \sum_{j=1}^{\nu-1}KV_0^{j} F W \bm g_{j}
=
\sum_{j=0}^{\nu-1}KV_0^{j} F W \bm g_{j},\nonumber
\end{align}
since $I = MM^{-1} = LE+KF$. Furthermore,
\begin{align}
 \bm\sigma_{0}&= \sum_{j=1}^{\nu-1}KV_0^{j} F W \bm g_{j}
 \nonumber\\
 &= KV_0 F \bm \sigma_1.
 \nonumber
\end{align}
%
%
The boundary equation (\ref{InitialCondEq}),
together with \eqref{gensol2}
gives
\begin{align}
\nonumber
(I-B-A_1G) \bm x &=
(B-I) (L \bm y + \bm \sigma_0)
+ A_1 (LV_{1}^{-1}\bm y + \bm \sigma_{1})
+\bm g_0
\\
\label{ProofPoissFin}
\text{or}
\quad \quad 
(I-P_*)\bm x
&=
\left(
\left(B-I\right)\wh G +A_1
\right)
(\bm \sigma_{1} + LV_1^{-1}\bm y)
+\bm g_0.
\end{align}
To see this, we observe that $\wh GL=LV_1$ and $\wh G K = K V_0$
so that
$L\bm y = \wh G L V_{1}^{-1} \bm y$
on the one hand and that
\begin{equation*}
\bm \sigma_0 =
KV_0F \bm \sigma_{1} 
= \wh G KF \bm \sigma_{1}
= \wh G \bm \sigma_{1}
\end{equation*}
on the other hand since $FK=I$.

If the QBD is transient, then $G$ is sub-stochastic, the matrix $I-P_*$ is nonsingular, and the constraint \eqref{e:xtran} on $\bm x$ immediately results from \eqref{ProofPoissFin} while there is no constraint on $\bm y$.

If the QBD is recurrent, then $G$ is stochastic and $P_*$ 
is the transition matrix of an 
irreducible finite Markov process, so that \eqref{ProofPoissFin} is a
finite Poisson equation and
\eqref{ExpressX} follows, provided that the right-hand side is in the span of the columns of $I-P_*$, that is, 
provided that 
\begin{equation}\label{e:contr}
\bm \pi_{0}^{T}
\left(
\left(B-I\right)\wh G +A_1
\right)
(\bm \sigma_{1} + LV_1^{-1}\bm y)
+\bm \pi_{0}^{T}\bm g_0
=0,
\end{equation}
as $\bm \pi_0^{T}(I-P_*)=\bm 0$ by \eqref{e:pi0}.
We have
\begin{align*}
\bm \pi_{0}^{T}
\left(
\left(B-I\right)\wh G +A_1
\right)
&=
\bm \pi_{0}^{T}
A_1(I-G\wh G)
\quad \quad \text{as $\bm \pi_{0}^{T}(I-B-A_1G)=\bm 0$,}
\\
&= 
-\bm \pi_{0}^{T} W^{-1} \wh G 
\quad \quad \text{by \eqref{Charact1H0},}
\end{align*}
and,
as we have seen earlier that $\wh GLV_{1}^{-1}=L$, the constraint \eqref{e:contr} may be written as
\begin{equation}\label{e:contr2}
\bm \pi_{0}^{T} W^{-1} L\bm y
= \bm \pi_{0}^{T} \bm g_0 -\bm \pi_{0}^{T}W^{-1} \wh G \bm \sigma_{1}.
 \end{equation} 

Now, having assumed that the series converges, we have
\begin{align}
 \sum_{k=0}^{\infty}R^k \bm g_k
&=
 \bm g_0
+ 
\sum_{k=1}^{\infty}W^{-1}\wh G^{k}W \bm g_k
\qquad \qquad \text{by \eqref{RelGandR},}
\label{e:un}
\\
&=
 \bm g_0
+
 W^{-1} \wh G
\sum_{k=0}^{\infty} K V_0^k F W \bm g_{k+1}
+
 W^{-1} \wh G
\sum_{k=0}^{\infty} L V_1^k E W \bm g_{k+1}
\nonumber
 \intertext{since $\wh G^k = KV_0^k F+LV_1^k E$,}
 &=
 \bm g_0
 -
  W^{-1} \wh G \bm \sigma_{1}
  -
    W^{-1} \wh G L V_{1}^{-1} \bm y^*.
    \label{e:utile}
\end{align}
Thus, \eqref{e:contr2} may be written as
\[
 \bm \pi_{0}^{T}  W^{-1} L \bm y
 = \bm \pi_0^{T} \sum_{k=0}^{\infty}R^k \bm g_k
 +  \bm \pi_{0}^{T} W^{-1} \wh G L V_{1}^{-1} \bm y^*
=
\bm \pi^{T} \bm g 
  +  \bm \pi_{0}^{T} W^{-1}  L  \bm y^*.
\]
This proves \eqref{ExpressY}.
\end{proof}

The expression of $\boldsymbol{u}_r$ given in \eqref{OurSolution} can be equivalently rewritten in a numerically more convenient form as follows
\begin{equation}\label{eq:numeric}
\boldsymbol{u}_r=G^r\boldsymbol{x}-\sum_{k=0}^{r-1} G^kW\boldsymbol{g}_{r-k}+
LV_1^{-r}\left(\boldsymbol{y}+\sum_{k=1}^r V_1^rEW\boldsymbol{g}_k \right)-
\sum_{j=1}^{\nu-1}KV_0^jFW\boldsymbol{g}_{j+r}.
\end{equation}

Some simplification occurs when  the matrix $A_1$ is nonsingular: then
the
matrix $R$ is nonsingular and the expression for the general solution simplifies as follows.

\begin{corollary}
Assume that $\det A_1\ne 0$.  Let $G$ and $R$ be the minimal
nonnegative solutions of \eqref{Geq} and \eqref{Req}.  The general
solution of the Poisson equation \eqref{PoissonEqGen} is given by
\[
\boldsymbol{u}_{r}
=
G^{r} \boldsymbol{x}
+
WR^{-r} \widetilde{\boldsymbol{y}}
-\sum_{k=1}^{r}
\left(
G^{r-k}W-WR^{k-r} 
\right)\boldsymbol{g}_{k},
\quad r\geq 0,
\]
where the vectors $\bm x$ and $\widetilde{\bm y}$ satisfy the following constraints.

If the QBD is transient, then
 $\widetilde{\boldsymbol{y}}\in \mathbb C^{m}$ is arbitrary  and
\[
\boldsymbol{x}
= (I-P_{*})^{-1}
\left((B-I)W \widetilde{\boldsymbol{y}}+A_{1}WR^{-1} \widetilde{\boldsymbol{y}}+ 
\boldsymbol{g}_{0}
\right),
\]
with $P_*=B+A_1G$.

If the QBD is positive recurrent and the series $\sum_{k=0}^{\infty}R^k \bm g_k$ converges, then 
$\widetilde{\boldsymbol{y}}= \widetilde{\bm y}^*+ \widetilde{\bm y}_{\perp}$ and
\[
\boldsymbol{x}
= (I-P_{*})^{\#}
\left((B-I)W \widetilde{\boldsymbol{y}}+A_{1}WR^{-1} \widetilde{\boldsymbol{y}}+ 
\boldsymbol{g}_{0}
\right)
+ \alpha \bm 1,
\]
where 
$\widetilde{\bm y}^*=- \sum_{k=1}^{\infty}R^k \bm g_k$,
$\widetilde{\bm y}_{\perp}$ is any vector in the hyperplane 
$\bm \pi_{0}^{T} \widetilde{\bm y}_{\perp} = \bm \pi^{T} \bm g$,
and $\alpha$ is arbitrary.
\end{corollary}

\begin{proof}
Since $R$ is nonsingular, $\wh G$ is nonsingular as well, $V_1$ is an
$m\times m$ matrix, $V_0$ does not exist, and we may take $M=I$,
$V_1=\wh G$, and $L=E=I$. In view of \eqref{RelGandR}, we have $V_1^{k}=\wh
G^{k}=WR^kW^{-1}$ for any integer $k$.
With this choice of matrices, we have
\[
V_1^{k-r}EW{\boldsymbol g}_k=WR^{k-r}W^{-1} W{\boldsymbol g}_k=WR^{k-r}{\boldsymbol g}_k
\]
and \eqref{sigmar} becomes
\[
\boldsymbol{\sigma}_r=-\sum_{k=1}^r(G^{r-k}W-WR^{k-r})\boldsymbol{g}_k.
\]
Set $\tilde{\boldsymbol{y}}=W^{-1} \boldsymbol{y}$ so that
$V_1^{-r}\boldsymbol{y}=WR^{-r}\tilde {\boldsymbol{y}}$. Replace the
latter expression in \eqref{OurSolution} and get
\[
\boldsymbol{u}_r=G^r\boldsymbol{x}+WR^{-r}\tilde{\boldsymbol{y}}+\boldsymbol{\sigma}_r.
\]
The reminder of the proof results from (\ref{e:xtran}, \ref{ExpressY}--\ref{eq:y}).
%
%
\end{proof}

The expression of $\boldsymbol{u}_r$ given in the above corollary can be equivalently rewritten in a numerically more convenient form as follows
\[
\boldsymbol{u}_r=G^r\boldsymbol{x}-\sum_{k=0}^{r-1} G^kW\boldsymbol{g}_{r-k}+
WR^{-r}\left(\boldsymbol{y}+\sum_{k=1}^r R^k\boldsymbol{g}_k \right).
\]

To conclude this section, we briefly examine the asymptotics of $\bm
u_r$ in \eqref{OurSolution} as $r \rightarrow \infty$ and we discuss
the effect of the powers of $G$ and of $V_1^{-1}$ for different
choices of $\boldsymbol{y}$ --- note that $\bm x$ is actually a function of the
arbitrary vector $\bm y$.

The powers of $G$ are bounded, since
$G\boldsymbol{1}\le\boldsymbol{1}$ and so, $\lim_{r\rightarrow \infty}
G^r \bm x$ is bounded for any given $\bm y$.  Concerning the powers of
$V_1^{-1}$, recall that the eigenvalues of $V_1$ coincide with the
nonzero eigenvalues of $R$. Thus, in the positive recurrent case where
the spectral radius $\rho(R)$ of the matrix $R$ is such that
$\rho(R)<1$, all eigenvalues of $V_1^{-1}$ are strictly greater than
one in absolute value and the powers of $V_1^{-1}$ diverge.
In the transient case where $\rho(R) = 1$, the powers of $V_1^{-1}$ diverge as well if $p>1$.
The term in the general solution \eqref{eq:numeric} 
 which involves $V_1$ is
\[
\bm s_r =LV_1^{-r}\left( \boldsymbol{y}+\sum_{k=1}^r V_1^{k}EW\boldsymbol{g}_k\right).
\]
Assume that  the
series $\sum_{k=0}^{\infty} R^k \boldsymbol{g}_k $ is
convergent. Under this assumption, the series 
$\sum_{k=1}^\infty V_1^kEW\boldsymbol{g}_k$ is convergent as well. Thus,
choosing  
 $\boldsymbol{y}= \boldsymbol{y}^*$ from
\eqref{eq:y}
implies that 
\[\boldsymbol{s}_r
=
-LV_1^{-r}\sum_{k=r+1}^\infty  V_1^kEW\boldsymbol{g}_k
=
- L\sum_{k=r+1}^\infty V_1^{k-r}EW\boldsymbol{g}_k.\]
 Whence $\boldsymbol{s}_r$ is bounded.
  
We discuss further the significance of the vector $\bm
y^*$ in Section~\ref{SectionComp}.


\section{The general solution: null recurrent case}
   \label{s:null}
   If the QBD is null recurrent, then $\xi_m=\xi_{m+1}=1$ and we
   cannot directly apply the arguments of the previous
   section. Indeed, both $G$ and $R$ have spectral radius equal to
   one, the matrix $W$ in (\ref{MatrixH0}) Theorem \ref{ThmforZ} is
   not defined, and the standard triple which allowed us to build a
   solution cannot be constructed.

   However, after a suitable manipulation, we can transform the
   original difference equation into a new one where we can express
   the solution through a standard triple. This manipulation is based
   on the shift technique of \cite{hmr,blm:shift}, which enables us to
   construct a new matrix polynomial having the same eigenvalues as
   the original polynomial except for $\xi_m$ and $\xi_{m+1}=1$ which
   are replaced by zero or by infinity. The new quadratic matrix
   polynomial is associated with a new matrix difference equation
   which can be solved by means of the resolvent triples as in Section
   \ref{SectionSolution}. We will prove that from the solution of the transformed
   matrix difference equation we can recover the solution of the
   original equation.

This transformation can be performed in different ways, say, by
applying a left shift, or a right shift, or combining together the two
transformations.

In this section we recall the shift technique from \cite{blm:shift},
 while in the next
section we describe the transformation which relates the solutions of
the matrix difference equations obtained this way.

In addition to $G$, $R$, and $\wh G$, define $\wh H= A_0-I+A_{-1}\wh G$, 
together with $\wh R=A_{-1} (I-\wh H)^{-1}$; the matrix $\wh R$
coincides with the minimal nonnegative solution of the matrix equation
\[
A_{-1}+X (A_0-I)+ X^2 A_{1}=0.
\]

For a null recurrent QBD we have $\rho(G)=\rho(\wh G)=\rho(R)=\rho(\wh
R)=1$, and 1 is a simple eigenvalue in each case.  If $X$ is any
of the four matrices, denote by $\bm w_X$ and $\bm v_X$ a right and a
left nonnegative eigenvectors, respectively, of $X$ corresponding to
the eigenvalue 1.  The main results of \cite{blm:shift} concerning the
right and the left shift are as follows.

\begin{theorem}[Right shift]\label{thm:rs}
Take $Q=\bm w_{G} \bm v_{\wh G}^T$, with  $\bm v_{\wh G}^T\bm
w_{G}=1$, and 
define 
\[
\wt A_{-1}= A_{-1}(I-Q), \qquad \wt A_0=A_0+A_1Q, \qquad \wt A_1=A_1.
\]
Normalize $\bm w_{\wh R}$ so that $\bm v_{\wh G}^T\wh H^{-1}\bm
w_{\wh R}=-1$.

The matrix equations 
\[
\wt A_{-1}+(\wt A_0-I)X+\wt A_1X^2=0 \qquad \mbox{and} \qquad \wt A_{1}+(\wt A_0-I)X+\wt A_{-1}X^2=0
\]
have the solutions $\wt G=G-Q$ and $\ddot{G}=\wh G+(\bm w_G +\wh
H^{-1}\bm w_{\wh R}) \bm v_{\wh G}^T$, respectively. Moreover,
$\rho(\wt G)<1$ and $\rho(\ddot{G})=1$, $\det(I-\wt G\ddot G)\ne 0$, and 
the matrix 
\begin{equation*}
\wt W=\sum_{i=0}^\infty \wt G^i(U-I)^{-1}R^i
\end{equation*}
 is nonsingular.
\end{theorem}

We recall that in the above theorem, the scalar product 
$\bm v_{\wh G}^T\wh H^{-1}\bm
w_{\wh R}$ is always non zero \cite{blm:shift}.

\begin{theorem}[Left shift]\label{thm:ls}
Take $S=\bm w_{\wh R}\bm v_{R}^T$, with $\bm v_{R}^T\bm w_{\wh
  R}=1$ and
define
\[
\wt A_{-1}=A_{-1}, \qquad \wt A_{0}=A_0+SA_{-1},  \qquad \wt
A_{1}=(I-S)A_1.
\]
 Normalize $\bm v_{\wh G}$ so that $\bm v_{\wh G}^T\wh H^{-1}\bm w_{\wh R}=-1$.  

The matrix
 equations
\[
\wt A_{-1}+(\wt A_0-I)X+\wt A_1X^2=0,\qquad \mbox{and} \qquad \wt A_{1}+(\wt A_0-I)X+\wt A_{-1}X^2=0
\]
have the solutions $G$ and $\ddot{G}=\wh G+\wh
H^{-1}\bm w_{\wh R}\bm v_{\wh G}^T$, respectively. Moreover,
$\rho(G)=1$ and $\rho(\ddot{G})<1$. 
\end{theorem}

For the next developments, it is useful to reformulate the difference equation \eqref{matrixequation} in the following functional form
\begin{equation}\label{eq:ff}
\eta(\lambda)\boldsymbol{u}(\lambda^{-1})=
{\boldsymbol{k}}_{-1}\lambda^2+{\boldsymbol{k}}_0\lambda-\sum_{j\ge 1}\boldsymbol{g}_j\lambda^{-j+1},\end{equation}
where $\boldsymbol u(\lambda)=\sum_{i=0}^\infty \boldsymbol{u}_i\lambda^i$, 
with 
${\boldsymbol{k}}_{-1}=A_1\boldsymbol{u}_0$ and
 ${\boldsymbol{k}}_0=(A_0-I)\boldsymbol{u}_0+A_1\boldsymbol{u}_1$.

\subsection{Solution based on the right shift}
Define $\wt{\eta}(\lambda)=\eta(\lambda)(I-\frac1{1-\lambda}Q)$.  It follows from 
\cite{blm:shift} that
\[
\wt\eta(\lambda)=\wt A_{-1}+\lambda(\wt A_0-I)+\lambda^2\wt A_1,
\]
where the matrices $\wt A_i$, $i=-1,0,1$,  are defined in Theorem
\ref{thm:rs}.  We may associate with the matrix polynomial $\wt \eta(\lambda)$ the matrix difference equation
\begin{equation}\label{eq:difftilde}
\wt A_{-1}\wt{\bm u}_{r}+(\wt A_0-I)\wt{\bm u}_{r+1}+\wt A_1\wt{\bm u}_{r+2}=-\bm g_{r+1}
\end{equation}

Our goal is to express  the solutions $\wt{\bm u}_r$ of the above difference
equation by means of standard triples, using the solutions
$\wt G$ and $\ddot{G}$ given in Theorem \ref{thm:rs}, and to relate these solutions to the general solution of the original matrix difference equation \eqref{matrixequation}.

By Theorem \ref{thm:rs}, the matrices $I-\wt G\ddot{G}$ and $\wt W$ are both
nonsingular.  Knowing this, we follow the steps in Lemma
\ref{decompair} and Theorem \ref{ThmforZ} and obtain a resolvent
triple for $\wt\eta(\lambda)$.  We apply Theorem \ref{MDET} and obtain the
general solution of \eqref{eq:difftilde}.

The solutions of \eqref{eq:difftilde} and those of \eqref{matrixequation} are related in a simple manner. Observe that the product 
$\eta(\lambda)\boldsymbol{u}(\lambda^{-1})$ is such that
\[
\eta(\lambda)\boldsymbol{u}(\lambda^{-1})=\eta(\lambda)(I-\frac1{1-\lambda}Q)(I-\frac1{1-\lambda}Q)^{-1}{\bm u}(\lambda^{-1})=\wt\eta(\lambda)\wt{\boldsymbol{u}}(\lambda^{-1}),
\]
where
$\widetilde{\bm u}(\lambda)$
is defined by
\begin{equation}\label{eq:uut1}
{\bm u}(\lambda)=(I-\frac \lambda{\lambda-1}Q)\wt{\bm u}(\lambda).
\end{equation}
Since $\wt\eta(\lambda)\wt{\bm u}(\lambda^{-1})=\eta(\lambda)\bm u(\lambda^{-1})$, from \eqref{eq:ff} we deduce that 
\[
\wt\eta(\lambda)\wt{\boldsymbol{u}}(\lambda^{-1}) =
{\boldsymbol{k}}_{-1}\lambda^2+{\boldsymbol{k}}_0\lambda-\sum_{j\ge
  1}\boldsymbol{g}_j\lambda^{-j+1}.
\]
Multiplying both sides of \eqref{eq:uut1} by $\lambda-1$ and comparing the terms with the same degree in $\lambda$ yields
\begin{equation}\label{eq:uut}
{\bm u}_0=\wt{\bm u}_0,\quad {\boldsymbol{u}}_k=\wt{\boldsymbol{u}}_k+Q\sum_{i=0}^{k-1}
\wt{\boldsymbol{u}}_i,  \quad  k\ge 1.
\end{equation}
and so 
$\bm k_{-1} = \wt A_{1} \wt {\bm u}_0$
and
$\bm k_{0} = (\wt A_{0}-I)\wt {\bm u}_0+ \wt A_{1} \wt {\bm u}_1$,
that is, the vector sequence $\wt{\bm u}_r$ solves \eqref{eq:difftilde}.
This proves that we may recover
the general solution of the original equation from
\eqref{eq:uut}.

In  view of Lemma \ref{Lemma52}, the general solution of the matrix difference equation \eqref{eq:difftilde} may be expressed  in the following form
\[
\begin{split}
&\wt {\bm u}_r=\wt G^r{\bm x} +\wt L\wt V_1^{-r}{\bm y}+\wt {\bm\sigma}_r,\quad r\ge 0\\
&\wt{\bm\sigma}_r=-\sum_{k=1}^{r}
\left(
\wt G^{r-k}-\wt L\wt V_1^{k-r}\wt E 
\right)\wt W\boldsymbol{g}_{k}
-
\sum_{j=1}^{\nu-1}
\wt K\wt V_0^{j}\wt F\wt W\boldsymbol{g}_{j+r},\quad r\geq 0,
\end{split}
\]
for any vectors $\bm x$ and $\bm y$
 where
\[
\ddot G\wt M=\wt M\wt J,\quad \wt J=\begin{bmatrix}\wt V_1&0\\ 0&\wt V_0\end{bmatrix},\quad
\wt M=\left[ \wt L \;  \text{\rm{\textbrokenbar}} \; \wt K\right]
\]
with $\rho(\wt V_0)=0$, $\det \wt V_1\ne 0$, and
 $\wt L$ and $\wt K$ are matrices of size $m\times p$ and $m\times(m-p)$, respectively;
moreover
\[
\wt M^{-1}=\begin{bmatrix}\wt E\\ \wt F\end{bmatrix}
\]
where $\wt E$ and $\wt F$ are matrices of size $p\times m$ and $(m-p)\times m$, respectively.

Now it remains to  analyze the solution which satisfies the initial conditions \eqref{InitialCondEq}. 
To this end we assume that the series $\sum_{k=0}^\infty R^k \bm g_k$ is convergent.

Observe that, since $\bm u_0=\wt{\bm u}_0$, and $\bm u_1=\wt{\bm u}_1+Q \wt{\bm u}_0$, the initial condition \eqref{InitialCondEq} can be rewritten as
\[
(\wt B-I)\wt {\bm u}_0+A_1\wt{\bm u}_1=-\bm g_0,\quad \wt B=B+A_1Q.
\]
Rewriting the above equation in terms of the vectors $\bm x$ and $\bm y$ and exploiting the identities $\wt G=G-Q$, $\wt L\bm y=\ddot G\wt L\wt V_1^{-1} \bm y$  and $\wt{\bm \sigma}_0=\ddot G\wt{\bm\sigma}_1$ yields
\[
(I-B-A_1G)\bm x=\bm g_0+\left( (\wt B-I)\ddot G+A_1\right)(\wt{\bm \sigma}_1+\wt L\wt V_1^{-1}\bm y). 
\]
The matrix $P_*=B+A_1G$ is stochastic and $\bm \pi_0^T (I-P_*)=0$ so that the above system has a solution if and only if the following condition is satisfied
\begin{equation}\label{eq:condexist}
\bm\pi_0^T\bm g_0+\bm\pi_0^T\left( (\wt B-I)\ddot G+A_1)\right)(\wt{\bm \sigma}_1+\wt L\wt V_1^{-1}\bm y)
=0.
\end{equation}
Observe that $P_*=I-\wt B - A_1 \wt G$ and therefore  
 $\bm\pi_0^T(I-\wt B-A_1\wt G)=0$
 from which we get $\bm\pi_0^T ( (\wt B-I)\ddot G+A_1))=\bm\pi_0^TA_1(I-\wt G\ddot G)$.  Since $\det (I-\wt G\ddot G)\ne 0$,  we may proceed as in the proof of Theorem \ref{mainresult} and arrive at the following equivalent formulation of condition \eqref{eq:condexist}
\[
\bm\pi_0^T\wt W^{-1}\wt L(\bm y-\wt{\bm y}^*) =\bm\pi^T\bm g,
\]
where 
$\wt{\bm y}^*=-\sum_{k=0}^\infty \wt V_1^k\wt E\wt W\bm g_k$.
 Observe that the definition of $\wt {\bm y}^*$ is consistent since  we assumed that $\sum_{k=0}^\infty R^k \bm g_k$ is finite. The latter property implies also that $\bm\pi^T\bm g$ is finite.

\subsection{Solution based on the left shift}

Instead of shifting $\xi_m$ to 0 (or equivalently, replacing the
maximal eigenvalue of $G$ by 0), we may shift $\xi_{m+1}$ to $\infty$,
and so shift the maximal eigenvalue of $\wh G$ to $1/\xi_{m+1}=0$.
This requires us to use the left shift and to modify the right-hand
side of (\ref{matrixequation}).

Rewrite \eqref{eq:ff} in the following form:
\begin{equation}
   \label{e:tb}
(I-\frac \lambda{\lambda-1}S)\eta(\lambda) \boldsymbol{u}(\lambda^{-1})=
(I-\frac \lambda{\lambda-1}S)({\boldsymbol{k}}_{-1}\lambda^2+{\boldsymbol{k}}_0\lambda-\sum_{j\ge 1}\boldsymbol{g}_j\lambda^{-j+1}),
\end{equation}
where $S=\boldsymbol{w}_{\wh R}\boldsymbol{v}_R^T$,
$\boldsymbol{v}_R^T\boldsymbol{w}_{\wh R}=1$,
and get
\[
\wt\eta(\lambda)\boldsymbol{u}(\lambda^{-1})=\wt{\boldsymbol{g}}(\lambda)
\]
with $\wt\eta(\lambda)=(I-\frac \lambda{\lambda-1}S)\eta(\lambda)$ and $\wt{\bm g}(\lambda)$ is the
right-hand side of (\ref{e:tb}).

Recall that $R$ is the minimal nonnegative solution of (\ref{Req}) and
that $\bm v_{R}$ is its left eigenvector corresponding to the
eigenvalue 1.  Thus,
\[
S(A_1 + (A_0-I) +A_{-1}) = S(A_1 + R (A_0-I) + R^2 A_{-1}) = 0
\]
and so $S(A_0-I) = -S(A_1 +A_{-1})$.   One readily verifies that 
$\wt\eta(\lambda)=\wt A_{-1}+\lambda(\wt A_0-I)+\lambda^2\wt A_1$, where 
$\wt A_{i}$, $i=-1,0,1$ are defined in Theorem \ref{thm:ls}.

If the function $\sum_{j\ge 1}\boldsymbol{g}_j\lambda^{-j+1}$ is analytic
for $|\lambda^{-1}|<c$ for some $c> 0$, since 
$ \lambda/(\lambda-1)= 1/(1-\lambda^{-1})$ is analytic for $|\lambda^{-1}|<1$, 
then 
the function $\wt{\boldsymbol{g}}(\lambda)$ is analytic for $|\lambda^{-1}|<\min (c,1)$,
and  we may write
$\wt{\boldsymbol{g}}(\lambda)=\sum_{i=-2}^{+\infty}\wt{\boldsymbol{g}}_i\lambda^{-i}$, where the coefficients
$\wt{\boldsymbol{g}}_r$ can be explicitly expressed as functions of $S$, ${\boldsymbol{k}}_{-1}$, ${\boldsymbol{k}}_{0}$,  and $\boldsymbol{g}_j$ for $j\ge 1$.

Thus, we obtain a modified difference equation in the form
\[
\wt A_{-1}{\boldsymbol{u}}_{r-1}+(\wt A_0-I){\boldsymbol{u}}_r+\wt
A_1\boldsymbol{u}_{r+1} =\wt{\boldsymbol{g}}_r.
\]
 In view of Theorem \ref{thm:ls}, the matrix equations associated
 with the matrix polynomial $\wt \eta(\lambda)$ have solutions $\wt G$ and
 $\ddot{G}$ such that $I-\wt G \ddot{G}$ is nonsingular, we may
 apply the technique of standard triples and obtain the explicit solution
 of the difference equation.

Observe that in this case, unlike the shift to the right, the solution of the transformed matrix difference  equation coincides with the solution of the original equation. Thus we do not need to reconstruct one solution from the other one.
On the other hand, with the shift to the left we have to compute a different right-hand side.


The two techniques of shifting to the right and to the left can be combined together to obtain another possible representation of the solution. We leave the details to the reader.


\section{Comparison with published results} \label{SectionComp}

Solutions of the Poisson equation are constructed in
\cite{dendievel2013poisson} under the assumptions that the QBD is
positive recurrent,
 that $\bm\pi_0^T \sum_{k=0}^\infty R^k \|\bm
g_k\| < \infty$,
and that $\bm\pi_0^T \sum_{k=0}^\infty R^k \bm g_k =0$.
The approach there is based on a probabilistic
argument, and a particular solution, up to an additive constant, is  written,
for $r\geq 0$, as 
\begin{align}\label{solYGS}
{\boldsymbol{\omega}}_{r}
&= G^{r}{\boldsymbol{\gamma}}+{\boldsymbol{y}}_{r}+ c \boldsymbol{1}
,
\end{align}
where $c$ is any arbitrary constant,
\begin{align}\label{eq:gamma}
{\boldsymbol{\gamma}}
&= (I-P_{*})^{\#} \sum_{k=0}^{\infty}R^{k}
\boldsymbol{g}_{k} 
\text{\qquad  and  \qquad}
{\boldsymbol{y}}_{r}
 =- \sum_{k=0}^{\infty}\sum_{j=0}^{r-1}G^{j}\left(U-I\right)^{-1}R^{k}
\boldsymbol{g}_{r+k-j}.
\end{align}

This solution has a different aspect from \eqref{OurSolution};
in particular, the right-hand sides of \eqref{eq:gamma} are expressed as
series while we have finite sums only in \eqref{OurSolution}, which is
more convenient for computational purposes. 
We show below that, for positive recurrent QBDs, the solution obtained
by probabilistic reasoning is identical to the solution from
Theorem \ref{mainresult}, 
for the specific choice of $\boldsymbol{y}=\bm y^*$
defined  in \eqref{eq:y}.  This is proved in Lemma~\ref{t:LEM}.   What
is more, we show in Lemma~\ref{t:nullrec} that the vectors $\bm
\omega_r$ actually form a solution of the Poisson equation for
transient or null recurrent QBDs also.

Our condition throughout is that  
$\sum_{k=0}^{\infty}
\rho(R)^k \|\boldsymbol{g}_k\| $
 should be a convergent series for some vector norm $\|\cdot\|$, and then automatically for any vector norm.  For
transient and for null recurrent QBDs, $\rho(R) =1 $ and this imposes a strong constraint on $\bm g$.
One immediate advantage stemming from
the constraint is that the series in (\ref{eq:gamma}) are all
convergent, and so the $\bm y_r$s are all well defined.  To see this,
we choose a vector norm such that $\|R\| = \rho(R)$ and we write
\begin{align*}
\|\bm y_r\| & \leq \|U-I\|^{-1} \sum_{k=0}^{\infty}\sum_{j=0}^{r-1}
\|G\|^{j} \|R\|^{k} \|\boldsymbol{g}_{r+k-j}\|  \\
 & = \|U-I\|^{-1} \sum_{k=0}^{\infty}\sum_{j=0}^{r-1}
\|G\|^{j} \rho(R)^{k} \|\boldsymbol{g}_{r+k-j} \| \\
 & = \|U-I\|^{-1} \sum_{j=0}^{r-1}    \|G\|^{j} \rho(R)^{j-r} \sum_{k=r-j}^{\infty}
\rho(R)^{k} \|\boldsymbol{g}_k\|
\end{align*}
which converges by assumption.

\begin{lemma}
   \label{t:LEM}
  Assume that the QBD is positive recurrent, that
  $\sum_{k=0}^{\infty} \rho(R)^k\|\boldsymbol{g}_k\| < \infty$ and that
  $\bm\pi_0^T \sum_{k=0}^\infty R^k \bm g_k =0$.
  Equation \eqref{solYGS} may be written as
\[
{\boldsymbol{\omega}}_{r}
=
G^r \boldsymbol{ x} 
+LV_1^{-r} \boldsymbol{y}^*
+ \boldsymbol{\sigma}_{r}
\]
where $\boldsymbol{y}^*$ is defined in \eqref{eq:y},
$\boldsymbol{x}$ is defined in \eqref{ExpressX} with $\bm y$ replaced
by $\bm y^*$,
and
$\bm \sigma_{r}$ is defined in \eqref{sigmar}.
\end{lemma}

\begin{proof} 
 We write
\begin{align*}
\boldsymbol{y}_{r}  &=
- \sum_{k=0}^{\infty}  \sum_{j=0}^{k+r-1} G^{j}(U-I)^{-1}R^{k}\boldsymbol{g}_{k+r-j}
+ \sum_{k=0}^{\infty}  \sum_{j=r}^{k+r-1}  G^{j}(U-I)^{-1}R^{k}\boldsymbol{g}_{k+r-j}
\\ 
& = - \sum_{k=1}^{\infty}\sum_{j=0}^{k-1}G^{j}(U-I)^{-1}R^{k}\boldsymbol{g}_{k+r-j}-\sum_{k=0}^{\infty}\sum_{j=k}^{k+r-1}G^{j}(U-I)^{-1}R^{k}\boldsymbol{g}_{k+r-j}
 +G^{r}\boldsymbol{\zeta}
\end{align*}
where
\[
\boldsymbol{\zeta}  
=\sum_{k=1}^{\infty}\sum_{j=0}^{k-1}G^{j}(U-I)^{-1}R^{k}\boldsymbol{g}_{k-j}
= \sum_{j=0}^\infty G^j (U-I)^{-1} \sum_{k=1}^\infty R^{k+j} \bm g_k
= W \sum_{k=1}^\infty R^k \bm g_k.
\]
We simplify the first term as
\[
\begin{split}
\sum_{k=1}^{\infty}\sum_{j=0}^{k-1}G^{j}(U-I)^{-1}R^{k}\boldsymbol{g}_{k+r-j} 
& = \sum_{k=1}^{\infty}
\sum_{i=0}^{\infty}
G^{i}(U-I)^{-1}R^{i+k}\boldsymbol{g}_{k+r}
\\
&= \sum_{k=r+1}^{\infty}WR^{k-r}\boldsymbol{g}_{k}
\\
  &=  \sum_{k=r+1}^{\infty}\wh{G}^{k-r}W\boldsymbol{g}_{k},
\end{split}
\]
by definition of $W$ and \eqref{RelGandR}.  The second term becomes
\begin{align*}
\sum_{k=0}^{\infty}\sum_{j=t}^{k+r-1}G^{j}(U-I)^{-1}R^{k}\boldsymbol{g}_{k+r-j}
 & =  \sum_{k=1}^{r}\sum_{t=0}^{\infty}G^{r-k+t}(U-I)^{-1}R^{t}\boldsymbol{g}_{k}\\
 & =  \sum_{k=1}^{r}G^{r-k}W\boldsymbol{g}_{k}.
\end{align*}
Thus,
\begin{equation*}
\boldsymbol{y}_{r}
 = - \sum_{k=r+1}^{\infty}\wh{G}^{k-r}W\boldsymbol{g}_{k}
 -\sum_{k=1}^{r}G^{r-k}W\boldsymbol{g}_{k}
 +G^{r}\boldsymbol{\zeta}.
\end{equation*}
By (\ref{JordanDecompGHat}--\ref{GhKW}, \ref{PartitionMm1}),   since $V_0^{j}=0$ for, $j\geq\nu$, we may 
write
\[
\begin{split}
\boldsymbol{y}_{r} 
 &= - LV_1^{-r}\sum_{k= 1}^{\infty}V_1^{k}EW\boldsymbol{g}_{k}
 +\sum_{k=1}^{r}LV_1^{k-r}EW\boldsymbol{g}_{k}
 -\sum_{j=1}^{\nu-1}KV_0^{j}FW\boldsymbol{g}_{j+r}
\\
 & \qquad-\sum_{k=1}^{r-1}G^{r-k}W\boldsymbol{g}_{k}
 +G^{r}\boldsymbol{\zeta} 
\\
 & = L V_1^{-r} \bm y^* + \bm\sigma_r + G^r \bm\zeta 
\end{split}
\]
and so 
\[
{\bm \omega}_r = G^r (\bm\gamma + \bm\zeta) + L V_1^{-r} \bm y^* +
\bm\sigma_r + c \bm 1.
\]
Finally, we verify that
$G^r \bm x = G^r (\bm\gamma + \bm\zeta)+ c_1\bm 1$, for some scalar $c_1$.
The equation \eqref{e:utile}
may be written as
\begin{equation*}
-W^{-1}\wh G (\bm \sigma_{1}+ LV_1^{-1}\bm y^*)
=\sum_{k=1}^{\infty}R^k \bm g_k
= W^{-1}\bm \zeta.
\end{equation*}
The vector $\bm x$  given in \eqref{ExpressX} may be rewritten as
\begin{align}
\bm x
&= 
(I-P_{*})^{\#}
\left(
-(B-I)\bm \zeta +
 A_1
\left( \bm \sigma_{1} + LV_1^{-1}\bm y^*
\right)
+\bm g_0
\right)
+\alpha \bm 1.
\label{e:utile2}
\end{align}
Furthermore, by repeating the argument 
(\ref{e:un}--\ref{e:utile}), we find that
\begin{equation*}
\bm \sigma_{1} + LV_{1}^{-1} \bm y^*
=
-W \sum_{k=0}^{\infty}R^k \bm g_{k+1}
\end{equation*}
and so
\begin{align*}
A_{1}(\bm \sigma_{1} +LV_1^{-1} \boldsymbol{y^*} )
&= (R-A_1GWR) \sum_{k=0}^\infty
R^k \bm g_{k+1} 
\quad\quad\quad \text{by Lemma \ref{LemmaInVDbl}}
\\
& = \sum_{k=1}^\infty
R^k \bm g_k - A_1 G \bm\zeta,
\end{align*}
and
\eqref{e:utile2}
becomes
\begin{align*}
\bm x &= 
(I-P_*)^{\#} 
\left((I-B-A_1G)\bm \zeta
+ \sum_{k=0}^{\infty} R^k \bm g_k\right)
+\alpha \bm 1
\\
&=
\left(I-(\bm \pi_0^T \bm1)^{-1} \bm 1 \bm \pi_0^T\right)
\bm \zeta
+ \bm \gamma
+\alpha \bm 1
\qquad\qquad \text{by \eqref{eq:groupinv}}
\\
&= \bm \zeta + \bm \gamma+c_2 \bm 1
\end{align*}
where $c_2$ is a scalar.
Since $G$ is stochastic, this completes the proof.
\end{proof}
%


To prove that the vectors $\bm \omega_r$ are always a solution of
(\ref{PoissonEqGen}), even if the QBD is null recurrent or transient,
we may not refer to the vector $\bm y^*$ since the definition
\eqref{eq:y} depends on the matrix $W$, and the series in
\eqref{MatrixH0} diverges in the null recurrent case.  Instead, we
prove by direct verification that (\ref{solYGS}) is a solution of
(\ref{InitialCondEq}, \ref{matrixequation}).

\begin{lemma}\label{t:nullrec}
Assume that 
 $\sum_{k=0}^{\infty}
\rho(R)^k \|\boldsymbol{g}_k\| $ converges,
where $R$ is the minimal nonnegative solution of \eqref{Req}.

If the QBD is recurrent, assume in addition that
$\bm \pi_{*}^{T} \sum_{k=0}^{\infty} R^k \boldsymbol{g}_k=0$, where
$\bm \pi_{*}$ is the stationary distribution of $P_{*}$:
$\bm \pi_{*}^{T}(I-P_{*})=\bm 0$, $\bm \pi_{*}^{T}\bm 1 = 1$.

Under these assumptions, one solution of the Poisson equation \eqref{PoissonEqGen} is given by~\eqref{solYGS}.
\end{lemma}

\begin{proof}
For positive recurrent  QBDs, 
$\bm \pi_{0}$ is proportional to 
$\bm \pi_{*}$, and so  the statement immediately results from Lemma~\ref{t:LEM}.

For null recurrent QBDs, we have
\begin{align*}
(B-I){\bm \omega}_0 +A_1 {\bm \omega}_1
&=
(B-I)\bm\gamma + A_1 G \bm\gamma +A_1 \bm y_1
\\
&=
 (P_*-I) \bm\gamma 
-A_1  \sum_{k=0}^{\infty} (U-I)^{-1} R^k \bm g_{k+1}
 \\
&=
(P_*-I)(I-P_*)^{\#} \sum_{k= 0}^{\infty} R^k \bm g_k + \sum_{k= 1}^{\infty}  R^k \bm g_k
\\
&=
(\bm 1\bm\pi_{*}^{T}-I) \sum_{k= 0}^{\infty} R^k \bm g_k 
+ \sum_{k=1}^{\infty}  R^k \bm g_k
\qquad \text{by \eqref{eq:groupinv}}
\\
&=
- \bm g_0.
\end{align*}
By \eqref{eq:gamma}, 
$\bm y_r=\sum_{j=0}^{r-1}G^{j}\bm z_{r-j}$,
with
$\bm z_n = (I-U)^{-1}\sum_{k=0}^{\infty}R^k \bm g_{n+k}$,
so that
\begin{align*}
A_{-1}{\bm \omega}_r + (A_0-I){\bm \omega}_{r+1}
+&A_1 {\bm \omega}_{r+2}
=
A_{-1}\bm y_r+(A_0-I)\bm y_{r+1}+A_1\bm y_{r+2}
\\
=&
(A_{-1}+(A_0-I)G+A_{1}G^2)
\sum_{j=0}^{r-1}G^{j}\bm z_j
\\
& + (A_0-I+A_1 G)\bm z_{r+1}
+ A_1 \bm z_{r+2}
\\
 =& - \sum_{k=0}^{\infty}R^k \bm g_{r+1+k}
+ A_1 (I-U)^{-1}\sum_{k=0}^{\infty}R^{k} \bm g_{r+2+k}
\\
=& - \bm g_{r+1}
\end{align*}
by equations (\ref{RfctU},\ref{UfctG}).

For transient QBDs, $P_*$ is sub-stochastic, 
$(I-P_*)^{\#}=(I-P_*)^{-1}$
and the same calculation as above apply. 
\end{proof}
%


\begin{remark} \em
It is interesting to note that the solution obtained by probabilistic reasoning, in the case of positive recurrent QBDs, corresponds to the solution (\ref{OurSolution}, \ref{ExpressY}, \ref{ExpressX}) with $\bm y_\perp = \bm 0$; this is another emphasis placed on the role played by the vector $\bm y^*$.   It is remarkable, in addition, that we should have found solutions of the Poisson equation, even when the process is null recurrent or even transient.  This is another illustration of the nice properties stemming from the transition structure of QBDs.
\end{remark}


\subsubsection*{Acknowledgements}
Guy Latouche and Sarah Dendievel thank the Minist\`ere de la
Communaut\'e fran\c{c}aise de Belgique for supporting this research
through the ARC grant AUWB-08/13--ULB~5. Dario A. Bini and Beatrice
Meini thank GNCS of INdAM for supporting this research, and
acknowledge the financial support of Pisa University through the PRA
project ``Mathematical models and computational methods for complex
networks''.


\begin{thebibliography}{10}

\bibitem{asmussen1992queueing}
S.~Asmussen.
\newblock Queueing simulation in heavy traffic.
\newblock {\em Mathematics of Operations Research}, 17(1):84--111, 1992.

\bibitem{asmussen1994poisson}
S.~Asmussen and M.~Bladt.
\newblock {P}oisson's equation for queues driven by a {M}arkovian marked point
  process.
\newblock {\em Queueing Systems}, 17(1-2):235--274, 1994.

\bibitem{blm:book}
D.~A. Bini, G.~Latouche, and B.~Meini.
\newblock {\em Numerical methods for structured {M}arkov chains}.
\newblock Numerical Mathematics and Scientific Computation. Oxford University
  Press, New York, 2005.
\newblock Oxford Science Publications.

\bibitem{blm:shift}
D.~A. Bini, G.~Latouche, and B.~Meini.
\newblock Shift techniques for quasi-birth and death processes: canonical
  factorizations and matrix equations.
\newblock 2016.
\newblock Submitted for publication. ArXiv:1601.07717.

\bibitem{campbell2009generalized}
S.~L. Campbell and C.~D. Meyer.
\newblock {\em Generalized Inverses of Linear Transformations}, volume~56.
\newblock SIAM, 2009.

\bibitem{dendievel2013poisson}
S.~Dendievel, G.~Latouche, and Y.~Liu.
\newblock {P}oisson's equation for discrete-time quasi-birth-and-death
  processes.
\newblock {\em Performance Evaluation}, 70:564--577, 2013.

\bibitem{glynn1994poisson}
P.~W. Glynn.
\newblock {P}oisson's equation for the recurrent {M/G/1} queue.
\newblock {\em Advances in Applied Probability}, pages 1044--1062, 1994.

\bibitem{gohberg2009matrix}
I.~Gohberg, P.~Lancaster, and L.~Rodman.
\newblock {\em Matrix Polynomials}, volume~58.
\newblock SIAM, 2009.

\bibitem{glr11b}
M.~Govorun, G.~Latouche, and M.-A. Remiche.
\newblock Stability for fluid queues: characteristic inequalities.
\newblock {\em Stochastic Models}, 29:64--88, 2013.
\newblock \\ doi: 10.1080/15326349.2013.750533.

\bibitem{hmr}
C.~He, B.~Meini, and N.~H. Rhee.
\newblock A shifted cyclic reduction algorithm for quasi-birth-death problems.
\newblock {\em SIAM J. Matrix Anal. Appl.}, 23(3):673--691 (electronic),
  2001/02.

\bibitem{jiang2014poisson}
S.~Jiang, Y.~Liu, and S.~Yao.
\newblock {P}oisson's equation for discrete-time single-birth processes.
\newblock {\em Statistics \& Probability Letters}, 85:78--83, 2014.

\bibitem{latouche1999introduction}
G.~Latouche and V.~Ramaswami.
\newblock {\em Introduction to Matrix Analytic Methods in Stochastic Modeling},
  volume~5 of {\em ASA-SIAM Series on Statistics and Applied Probability}.
\newblock Siam, Philadelphia PA, 1999.

\bibitem{li2013mixing}
Q.-L. Li and J.~Cao.
\newblock A computational framework for the mixing times in the {QBD} processes
  with infinitely-many levels.
\newblock {\em arXiv:1308.4227}, 2013.

\bibitem{makowski2002poisson}
A.~M. Makowski and A.~Shwartz.
\newblock The {P}oisson equation for countable {M}arkov chains: {Probabilistic}
  methods and interpretations.
\newblock In {\em Handbook of {M}arkov Decision Processes}, pages 269--303.
  Springer, 2002.

\bibitem{meyer1975role}
C.~D. Meyer, Jr.
\newblock The role of the group generalized inverse in the theory of finite
  {M}arkov chains.
\newblock {\em Siam Review}, 17(3):443--464, 1975.

\bibitem{neuts1981matrix}
M.~F. Neuts.
\newblock {\em Matrix-Geometric Solutions in Stochastic Models: An Algorithmic
  Approach}.
\newblock The Johns Hopkins University Press, Baltimore, MD, 1981.

\end{thebibliography}

\end{document}